\documentclass[11pt]{article}

\usepackage[top=1in, bottom=1in, left=1in, right=1in]{geometry}
\usepackage{amssymb}
\usepackage{graphicx}
\usepackage{color}
\usepackage[normalem]{ulem}
\usepackage{amsmath, amsthm, amssymb}
\usepackage{enumerate}
\usepackage{float}
\usepackage{url}
\usepackage{caption}
\usepackage{tikz}
\usetikzlibrary{shapes,automata,positioning,arrows,fit}
\usepackage{mathtools}
\usetikzlibrary{snakes}
\usepackage{tikz-cd}
\usetikzlibrary{babel}

\newtheorem{theorem}{Theorem}[section]
\newtheorem{corollary}[theorem]{Corollary}

\newtheorem{lemma}[theorem]{Lemma}

\newtheorem{remark}[theorem]{Remark}

\theoremstyle{definition}
\newtheorem{definition}[theorem]{Definition}
\theoremstyle{definition}

\newcommand{\R}{\mathbb{R}}

\newcommand{\Z}{\mathbb{Z}}

\DeclareMathOperator{\sgn}{sgn}

\newcommand{\specialcell}[2][c]{\begin{tabular}[#1]{@{}c@{}}#2\end{tabular}}

\numberwithin{equation}{section}

\begin{document}

\title{Computing with reaction networks at input-independent speed: exponential and logarithmic functions}

\author{David F. Anderson \footnote{University of Wisconsin-Madison (anderson$@$math.wisc.edu)}, ~ Badal Joshi \footnote{Department of Mathematics, California State University San Marcos (bjoshi$@$csusm.edu)}, ~ and Tung D. Nguyen \footnote{Department of Mathematics, University of California Los Angeles (tungdnguyen@math.ucla.edu)}}

\date{}

\maketitle

\begin{abstract}
\noindent 
The concept of \textit{input-independent computational  speed} for chemistry-based analog computers was introduced in \cite{anderson2025arithmetic}, where it was shown that arithmetic operations can be computed  with a convergence rate bounded below independently of the input values. Here, by inputs we mean the numerical values encoded by the initial concentrations of designated input species, with the underlying reaction network and rate constants held fixed. Combining these operations via power series to approximate transcendental functions is possible in principle, but the number of chemical species required grows with the number of terms retained, and achieving sufficient accuracy may demand many terms---a burden that is especially severe for slowly converging series such as the power series for the logarithm.
In this paper, we begin the program of directly computing transcendental functions by chemical reaction networks by focusing on the exponential and logarithmic functions, two widely used transcendental functions, and constructing reaction network modules that compute them without relying on truncated power series.
We show that the resulting modules are mass-action systems, and prove that they achieve arbitrary accuracy given sufficient time while operating at input-independent speed. Moreover, these modules can be  combined with each other and with the arithmetic modules of~\cite{anderson2025arithmetic}: any such composite computation also runs at input-independent speed,  with a convergence rate that does not depend on the length of the feedforward chain. 
Logarithm and exponential functions serve as foundational cases, and the constructions developed here are intended to serve as templates for the direct computation of more general transcendental functions by chemical reaction networks.
~


\end{abstract}

\section{Introduction}

The concept of \textit{input-independent computational  speed} for a chemistry-based analog computer was introduced and articulated in \cite{anderson2025arithmetic}. 
The authors showed that an arithmetic computation of arbitrary complexity (i.e.\ consisting of any finite number of elementary steps) can be performed  at a speed bounded below, independent of the inputs to the computation as well as of the number of steps required for the computation. 
The proof relied on explicit constructions of reaction network modules for the elementary arithmetic operations. Each module was shown to operate at input-independent speed: the time required to compute $n$ decimal places is proportional to 
$n$ and independent of the input values. Moreover, when such modules are composed into a more complex computation, the elementary steps run in parallel, so  the convergence rate is controlled by the minimum intrinsic speed among the component modules rather than by serial waiting times.
This remains true not only when the computational steps are independent, but also when they are in ``series''---requiring the output of one step to be fed as input into another in a potentially long feedforward chain. 
\textit{Parallel processing at an elementary scale} is a striking feature of chemistry-based analog computers which can be leveraged for efficient and speedy computation. 

The results in \cite{anderson2025arithmetic} imply that arithmetic computations can be composed, using power series for instance, to approximate transcendental functions. Such a computation, however, requires a number of chemical species that grows with the number of terms retained in the series. Moreover, achieving high accuracy demands many terms, a burden that is especially severe for slowly converging series: the power series for the logarithm, for instance, converges only on the unit interval and does so slowly near the boundary, requiring a large number of terms to obtain even modest accuracy over a wide input range. In addition, the truncation level $m$ must be chosen in advance, so running the network longer improves convergence only to the chosen truncated series, not to the underlying transcendental function itself.

If chemistry-based analog computation is to serve as a viable computational paradigm, complex function evaluations must be achievable without requiring a large and growing number of chemical species.
The exponential and logarithmic functions are natural and important first test cases: they are foundational in mathematics and engineering, and their computation directly illuminates the challenges that arise for transcendental functions more broadly. 
Moreover, the two cases already display markedly different levels of difficulty. The exponential module developed here is relatively direct, whereas the logarithmic computation is substantially more subtle, owing in part to the fact that the logarithm may take negative values and therefore requires the dual-rail representation introduced in Section~\ref{sec:framework}. Thus, these examples are important not only because of the central role of the functions themselves, but also because they expose several of the main difficulties inherent in chemistry-based analog computation of transcendental functions. 
This motivates the design of reaction network algorithms that compute these functions directly over their entire domains, without relying on an approximation scheme. 
There are many possible ways to realize such computations chemically, but our goal is to do so in a principled and efficient manner. Accordingly, we insist that each reaction network module: (i) be \textit{efficient}, in the sense that it avoids unnecessary auxiliary species; (ii) achieve \textit{arbitrary accuracy} given sufficient computational time; (iii) compute accurately over the \textit{entire domain} of the function; and (iv) operate at \textit{input-independent speed}. 
Throughout, all reaction rate constants are taken to be $1$ for tractability. Extending these results to more general choices of rate constants is an important and difficult problem, and we expect that treating it properly will require new ideas.

We note here a departure from the constructions in \cite{anderson2025arithmetic}, where the intermediate and output species could be assigned arbitrary initial values. In the present work, by contrast, certain intermediate and output species must be initialized to specific values. Put differently, the desired computation is obtained only when these species are started in carefully balanced initial configurations; arbitrary initializations would generally alter the limiting behavior and hence the computed output.

Despite this difference, the new modules remain composable: finite feedforward combinations of the exponential and logarithm modules developed here with the arithmetic modules of~\cite{anderson2025arithmetic} run at input-independent speed.

Chemical reaction networks have been studied as computational devices from several
complementary perspectives.
In the discrete stochastic setting, Cook, Soloveichik, Winfree, and Bruck
\cite{cook2009programmability} gave a comprehensive account of the computational power
of stochastic CRNs,  Chen, Doty, and Soloveichik \cite{chen2014deterministic}
characterized the functions deterministically computable by such networks as precisely
the semilinear functions, and  Cappelletti, Ortiz-Mu\~noz, Anderson, and Winfree
\cite{cappelletti2020stochastic} showed that stochastic CRNs can robustly approximate arbitrary
discrete probability distributions in stationarity (thus providing a way to encode information).
In the continuous deterministic setting, Fages, Le\,Guludec, Bournez, and Pouly
\cite{fages2017strong} proved that continuous CRNs are strongly Turing-complete and
derived a compiler from mathematical functions into elementary reactions; Chen, Doty,
Reeves, and Soloveichik \cite{chen2023rateindependent} characterised rate-independent
continuous CRN computation as exactly the piecewise rational linear functions, using
the dual-rail representation that also appears in the present paper; and Anderson, Joshi,
and Deshpande \cite{anderson2021neural} developed a framework for implementing neural
networks via mass-action CRNs with exponentially fast convergence from any initial
condition.
On the implementation side, Soloveichik, Seelig, and Winfree \cite{soloveichik2010dna}
showed that DNA strand displacement reactions can implement arbitrary CRNs in vitro
\cite{qian2011scaling,qian2011neural}.

Prior work on computing the exponential and logarithm in analog and ODE-based settings
is directly relevant here.
Bournez, Campagnolo, Gra\c{c}a, and Hainry \cite{bournez2007polynomial} proved that
every function computable in the sense of computable analysis can be defined by a
polynomial ODE system over compact intervals, and Bournez, Gra\c{c}a, and Pouly
\cite{bournez2017ptime} established that polynomial time computation corresponds
precisely to polynomial-length solutions of polynomial ODEs, proving that the General
Purpose Analog Computer (GPAC) \cite{shannon1941} and Turing machines are equivalent in time complexity;
this gives a complexity-theoretic framework within which the efficiency of ODE-based
function evaluation---including the exponential and logarithm---can be meaningfully
situated.

Some of the polynomial ODE systems we use as intermediate constructions
(Systems~1 and~2 in Section~\ref{sec:log}) are not \emph{themselves} mass-action:
their vector fields violate the sign condition of \cite{hars1981inverse}, in that a
negative term of some $\dot x_i$ lacks a factor of $x_i$. This is a property of the
particular vector field, not a statement that the corresponding function fails to be
CRN-computable; indeed, by the dual-rail compilation of Fages, Le\,Guludec, Bournez, and
Pouly \cite{fages2017strong}, \emph{any} polynomial ODE system can be realized by a
mass-action CRN. Our aim here is not realizability as such, but \emph{direct} mass-action
constructions satisfying the requirements (i)--(iv) above, developed in
Section~\ref{sec:log} and culminating in System~6.
Daniel, Rubens, Sarpeshkar, and Lu \cite{daniel2013analog} demonstrated experimentally
that synthetic analog gene circuits in living \textit{E.\ coli} can compute logarithmic
functions over four orders of magnitude, exploiting the analogy between transistor and
chemical kinetics; Sarpeshkar \cite{sarpeshkar2014analog} analyzed the resource
efficiency of such computations from fundamental noise limits.
Chou \cite{chou2017logarithm} studied logarithm computation using abstract CRNs,
proposing a Pad\'e-approximation method that trades off reaction count against accuracy.
Chou asserts, without proof, that CRNs cannot compute the logarithm exactly.
Whether such an assertion holds depends on the notion of computation one adopts. Under
the notion used in this paper---convergence of a trajectory to
the exact value from a compatible initialization---it does not hold: we give several
mass-action CRN modules whose trajectories converge to the exact logarithmic value. There
is, however, a distinct and stronger sense in which the logarithm is not CRN-computable:
any function stabilized by a mass-action CRN---in the sense of absolute functional
robustness---must be algebraic \cite[Lem.~14]{hemery2024absolute}, and $\ln$ is not. We make
these framework distinctions precise in Section~\ref{sec:framework}
(Definition~\ref{def:computation} and
Remark~\ref{rem:frameworks}).
Moreover, whereas Chou's approach requires increasing the number of reactions and
intermediate species to improve accuracy, our constructions achieve arbitrary accuracy
with a fixed network simply by allowing more computational time.

In a related line of work on the \emph{speed} of CRN computation, Huang, Klinge,
Lathrop, Li, and Lutz \cite{huang2019computability} studied which real numbers are
computable by CRNs in \emph{real time}---a speed constraint requiring the output to
remain within $2^{-t}$ of the target for all $t \ge 1$, which is strictly stronger than
our notion of input-independent speed---and Huang, Klinge, and Lathrop
\cite{huang2019equivalence} established the equivalence of real-time CRN and GPAC
computation \cite{fletcher2025}.
The Hartmanis-Stearns conjecture \cite{hartmanis1965} asserts that no irrational
algebraic number is real-time computable by a Turing machine; by contrast, CRNs can
compute all algebraic numbers in real time \cite{huang2019computability}, highlighting
a fundamental difference between the two computational paradigms.
Finally, Anderson and Joshi \cite{anderson2025arithmetic} introduced
\textit{input-independent computational speed} for mass-action systems and showed that
elementary arithmetic operations can be performed  at a convergence rate bounded below independently of the input values; the present paper extends that program to the transcendental functions
$e^x$ and $\ln x$.

The paper is structured as follows: in Section~\ref{sec:framework}, we describe the mathematical framework and notation used throughout the paper. In Section~\ref{sec:exp} we construct and analyze a reaction network module for computing the exponential function. In Section~\ref{sec:log} we develop a sequence of logarithm modules---from simple polynomial systems through mass-action and dual-rail variants, culminating in System~5, which achieves input-independent speed for $a^*\ge 1$, and System~6, which achieves it on the full positive domain. In Section~\ref{sec:general_composite} we prove a general composition theorem showing that  finite feedforward combinations of the exponential and logarithm modules developed here with the arithmetic modules of~\cite{anderson2025arithmetic} run at input-independent speed. Finally, Section~\ref{sec:discuss} contains a summary and discussion of the results, together with some directions for future work.

\section{Mathematical framework}
\label{sec:framework}

The constructions in this paper are polynomial ODE systems. 
Our initial constructions are sometimes non-mass-action, but in each case we refine them to produce a mass-action reaction network. 
In the mass-action setting, each reaction proceeds at a rate proportional to the product of the concentrations of its reactant species, and trajectories remain in the nonnegative orthant, so concentrations never become negative.
A necessary and sufficient condition for a polynomial ODE system to be mass-action is that every negative term in the equation for $\dot x_i$ contains a factor of $x_i$ \cite{hars1981inverse}: chemically, this simply means that species $X_i$ can only be consumed if it participates as a reactant.  

Throughout, we use standard chemical reaction network notation. Chemical species are denoted by capital letters, such as $X$, $Y$, and $Z$, while their corresponding concentrations are denoted by the associated lower-case letters, such as $x(t)$, $y(t)$, and $z(t)$. When no confusion will arise, we suppress the explicit dependence on time.

As a simple example that is relevant to the work in this paper, consider the reaction network
\[
X \to X+Z,
\qquad
Y \to Y+Z,
\qquad
Z \to \emptyset.
\]
Here $X$ and $Y$ are input species and $Z$ is the output species. Under mass-action kinetics, the corresponding system of differential equations is
\[
\dot x = 0,
\qquad
\dot y = 0,
\qquad
\dot z = x+y-z.
\]
Thus the input concentrations remain constant in time, while the output species evolves according to a linear relaxation equation.

If the initial input concentrations are $x(0)=x_0$ and $y(0)=y_0$, then
\[
x(t)=x_0,
\qquad
y(t)=y_0
\]
for all $t\ge 0$, and hence
\[
\dot z = x_0+y_0-z.
\]
It follows that
\[
z(t) = x_0+y_0 + \bigl(z(0)-(x_0+y_0)\bigr)e^{-t}.
\]
In particular,
\[
\lim_{t\to\infty} z(t)=x_0+y_0.
\]
Thus the concentration of the output species $Z$ converges exponentially fast to the sum of the two input values. In this sense, the reaction network carries out the computation of addition.

The preceding example illustrates a feature that will be important throughout the paper: not only does the output converge to the desired value, but it does so at an exponential rate. This motivates the following definition.

\begin{definition}
\label{def:speed}
Let $u(t)$ be a real-valued function with $\lim_{t\to\infty} u(t)=u^*$.

Following~\cite{anderson2025arithmetic}, define the asymptotic convergence rate of $u(t)$ to $u^*$ by
\[
\operatorname{rate}(u\to u^*)
:=-\limsup_{t\to\infty}\frac{\ln|u(t)-u^*|}{t},
\]
where we use the extended-real convention $\ln 0=-\infty$. We say that $u(t)$ converges to $u^*$ with rate at least $\rho>0$ if $\operatorname{rate}(u\to u^*)\ge\rho$. Equivalently, for every $\eta\in(0,\rho)$ there exist constants $C_\eta>0$ and $T_\eta\ge0$ such that
\[
|u(t)-u^*|\le C_\eta e^{-(\rho-\eta)t}
\qquad\text{for all } t\ge T_\eta.
\]
In this case, we also say that the computation of $u^*$ by $u(t)$ has speed at least $\rho$. Indeed, exponential convergence with rate $\rho$ implies that the number of correct decimal digits grows linearly in time, with growth proportional to $\rho$.
\end{definition}

For the addition module above, we have
\[
z(t)-(x_0+y_0)=\bigl(z(0)-(x_0+y_0)\bigr)e^{-t},
\]
and hence
\[
|z(t)-(x_0+y_0)|\le Ce^{-t}
\]
with $C=\bigl|z(0)-(x_0+y_0)\bigr|$. Thus the output converges to the correct value with rate at least $1$, and so the computation has speed at least $1$. Importantly, this lower bound does not depend on the input values $x_0$ and $y_0$. For this reason, we say that the addition module computes at \textit{input-independent speed}.

By contrast, not every natural reaction network construction has this property. For example, if $A$ is an input species with constant concentration $a>0$, then the network
\[
\emptyset \to X,
\qquad
A+X \to A
\]
gives rise to the differential equation
\[
\dot x = 1-ax.
\]
The solution is $x(t) = \frac1a + (x_0-\frac1a)e^{-at},$ which converges to $1/a$ at rate $a$, so the speed of computation depends on the input value itself. In particular, as $a\to 0$, the convergence becomes arbitrarily slow. This is precisely the kind of input dependence that we seek to avoid.

In several of the constructions below, we will need to represent real numbers that may be either positive or negative while still working with nonnegative concentrations. For this purpose, we use the standard \emph{dual-rail representation}: a real number $r\in\R$ is encoded by a pair of nonnegative quantities $(r_p,r_n)$ and interpreted as
\[
r = r_p-r_n.
\]
This representation is not unique, but in many constructions one arranges that at most one of the two rails is active in the limit. In that case, $r_p$ records the positive part of the value and $r_n$ records the magnitude of the negative part. Dual rail will be used below both to allow arbitrary real inputs and to represent outputs, such as $\ln a^*$, that may be negative.

We conclude this section by making precise the notion of computation used throughout,
and situating it relative to two other notions in the literature.

\begin{definition}[Computation at a given speed]
\label{def:computation}
Let $f:D\to\R$ with $D\subseteq\R^k$. Consider a reaction network equipped with
designated \emph{input} species, whose concentrations $a_1(t),\dots,a_k(t)$ (possibly in
dual-rail form) may evolve in time; a designated \emph{output} species, or dual-rail pair,
whose value is interpreted as the result; and a specified set of \emph{admissible} initial
conditions for the output species and all remaining (intermediate) species. We say the
network \emph{computes $f$ at intrinsic speed at least $\rho>0$} if, for every
$a^*=(a_1^*,\dots,a_k^*)\in D$, every choice of $\rho_1,\dots,\rho_k>0$, every collection of input
trajectories $a_i(t)\to a_i^*$ converging at rates at least $\rho_1,\dots,\rho_k$
respectively, and every admissible initialization, the output
converges to $f(a^*)$ at rate at least $\min\{\rho_1,\dots,\rho_k,\rho\}$ in the sense of
Definition~\ref{def:speed}. If a single $\rho>0$ may be chosen uniformly over all $a^*\in D$, we say the network computes $f$ at
\emph{input-independent speed}.

We refer to such a network as a \emph{module} when it is used as a component of a
larger computation.
\end{definition}

\begin{remark}[Relation to other notions of CRN computation]
\label{rem:frameworks}
Definition~\ref{def:computation} fixes a set of admissible initial conditions and
asks for convergence to the exact value at a uniform rate. It should be distinguished from
two other notions. First, in the \emph{absolutely functionally robust} (AFR) model of Hemery
and Fages \cite{hemery2024absolute}, the output must converge to $f(X)$ from a
full-dimensional domain of admissible initial conditions for the intermediate and output species
(equivalently, the CRN \emph{stabilizes} $f$).
Their main theorem characterizes the functions computable in this sense as \emph{exactly}
the real algebraic functions. Since $\exp$ and $\ln$ are transcendental, they are
\emph{not} AFR-computable, a fact consistent with the initialization sensitivity of our
modules, which we analyze in Section~\ref{sec:init}. Second, \emph{real-time} computation in the sense of Huang, Klinge,
Lathrop, Li, and Lutz \cite{huang2019computability} imposes the stronger requirement that
the output remain within $2^{-t}$ of the target for all $t\ge 1$; input-independent speed
is the weaker requirement of a uniform exponential rate. Thus the modules of this paper
compute $\exp$ and $\ln$ exactly, and at input-independent speed, but not necessarily with
absolute robustness.
\end{remark}

\section{Computing the exponential function at input-independent speed}
\label{sec:exp}

We begin with the exponential function, which turns out to admit a relatively direct construction. Our goal is to compute $e^{a^*}$ by a mass-action reaction network, where the input value $a^*$ is encoded by a designated species. We first treat the case of a nonnegative input, $a^*\in\R_{\ge 0}$, for which a direct module is available. We then extend the construction to arbitrary real inputs by passing to the dual-rail representation introduced in Section~\ref{sec:framework} and composing the nonnegative exponential module with reciprocal and multiplication modules.

Consider first the case of a nonnegative input $a^*\in\R_{\ge 0}$. We use the reaction network
\begin{equation} \label{eq:exponential}
\begin{aligned}
\{A \to A+Z, \quad Z \to 0, \quad &A+X \to A+2X, \quad  Z+X \to Z\}, \\
\dot z(t) = a(t) - z(t), \qquad
&\dot x(t) = x(t) \left(a(t) - z(t)\right). 
\end{aligned}
\end{equation}
Note that the reactions leave $A$ unchanged. In general, the system in \eqref{eq:exponential} may be only a subsystem of a larger network, and so we allow the input concentration $a(t)$ to evolve in time and converge to a limit $a^*$. 

We begin, however, with the simpler case in which the input is fixed, i.e.\ $a(t)\equiv a^*$.
In this constant-input case we can write the solution explicitly. For the initial values $x(0)=1$ and $z(0)=0$, the solution for all $t\ge 0$ is
\begin{equation} \label{eq:exponential_sol}
\begin{aligned}
z(t) = a^* \left(1 - e^{-t}\right), \qquad x(t) = \exp \left( a^* \left(1 - e^{-t}\right) \right).
\end{aligned}
\end{equation}
Hence $x(t)\to e^{a^*}$ as $t\to\infty$. Moreover, the output converges to the correct value at rate at least $1$, independently of the input value $a^*$.

Now suppose that the input $a(t)$ evolves in time and converges to a limiting value $a^*$, i.e.\ $a(t)\to a^*$. Assume that the rate of this convergence is at least $\rho_a$, in the sense of Definition~\ref{def:speed}. Since
\[
\dot x(t)=x(t)\bigl(a(t)-z(t)\bigr),
\]
any solution with $x(0)>0$ satisfies $x(t)>0$ for all $t\ge 0$. Hence $\ln x(t)$ is well defined, and
\[
\frac{d}{dt}\ln x(t)=\frac{\dot x(t)}{x(t)}=a(t)-z(t)=\dot z(t).
\]
Thus the quantity $\ln x(t)-z(t)$ is constant in time. Integrating, we obtain
\[
\ln x(t)-\ln x(0)=z(t)-z(0).
\]
In particular, for any initial values satisfying $\ln x(0)=z(0)$, for instance $x(0)=1$ and $z(0)=0$, we have
\[
\ln x(t)=z(t)
\qquad\text{for all } t\ge 0,
\]
and therefore
\[
x(t)=e^{z(t)}
\qquad\text{for all } t\ge 0.
\]

We may now state and prove the main result for the nonnegative exponential module. The key point is that, once the initial values are chosen so that $\ln x(t)=z(t)$ for all $t\ge 0$, the convergence of $x(t)$ to $e^{a^*}$ is determined entirely by the convergence of $z(t)$ to $a^*$.

\begin{theorem}[Nonnegative exponential at input-independent speed]
\label{thm:exponential}
Consider the mass-action system in \eqref{eq:exponential} with initial conditions satisfying $\ln x(0) = z(0)$, where $a(t)$ is a non-negative-valued function of time that converges to a non-negative constant $a^*$ at rate at least $\rho_a$. 
Then $x(t) \to e^{a^*}$ at a rate of at least $\min\{\rho_{a},1\}$.    
\end{theorem}
\begin{proof}
By Lemma~\ref{thm:system1_analysis} (applied with $g_1(t) = a(t)$ and $g_2(t) \equiv 1$), $z(t) \to a^*$ at rate at least $\rho_z = \min\{\rho_a,1\}$. 
Since $x(t) = e^{z(t)}$ for all $t \ge 0$ and the exponential function is locally Lipschitz on $\R$, Lemma~\ref{lem:lipschitz} gives $x(t) \to e^{a^*}$ at rate at least $\rho_z = \min\{\rho_a, 1\}$.
\end{proof}

\begin{remark}[Initialization sensitivity]\label{rem:init}
Theorem~\ref{thm:exponential} requires the initial conditions to satisfy $\ln x(0) = z(0)$. This is not a mere convenience: if the condition fails, the identity $x(t) = e^{z(t)}$ no longer holds, and the output converges to $e^{\ln x(0)-z(0)} \cdot e^{a^*}$ rather than $e^{a^*}$---a persistent multiplicative error determined entirely by the initialization mismatch. The same sensitivity arises in every subsequent construction in this paper: the real exponential module of Theorem~\ref{thm:real_exponential} requires $\ln x(0)=z(0)$ for each of its two sub-modules, and the logarithm modules of Systems~1, 2, 3, 4p, and~4n each depend on an analogous compatibility condition between initial values. This stands in contrast to the arithmetic modules of \cite{anderson2025arithmetic}, where the output converges to the correct value from any admissible strictly positive initial condition.
\end{remark}

We now extend the construction to arbitrary real inputs. Let $a\in\R$ and write
\[
a=a_p-a_n,
\qquad a_p,a_n\ge 0,
\]
so that $a$ is represented in dual-rail form by the pair $(a_p,a_n)$. Then
\[
e^a=e^{a_p-a_n}=\frac{e^{a_p}}{e^{a_n}}.
\]
Thus the exponential of a real input can be computed by composing two instances of the nonnegative exponential module with a reciprocal module and a multiplication module. Input-independent constructions for reciprocal and multiplication were given in \cite{anderson2025arithmetic}, and the resulting composition again preserves input-independent speed.

\begin{theorem}[Real exponential at input-independent speed]
\label{thm:real_exponential}
Let
\[
a(t)=a_p(t)-a_n(t)
\]
be a real-valued input converging to
\[
a^*=a_p^*-a_n^*\in\R,
\]
where $a_p(t)$ and $a_n(t)$ are non-negative-valued functions of time converging to non-negative limits $a_p^*$ and $a_n^*$ each at rate at least $\rho_a>0$. Consider the composition of:
\begin{enumerate}
    \item two instances of the nonnegative exponential module \eqref{eq:exponential}, computing $e^{a_p^*}$ and $e^{a_n^*}$, each initialized with $\ln x(0)=z(0)$;
    \item a reciprocal module computing $1/e^{a_n^*}$;
    \item a multiplication module computing
    \[
    e^{a_p^*}\cdot \frac{1}{e^{a_n^*}}=e^{a^*}.
    \]
\end{enumerate}
Then the overall system computes $e^{a^*}$ at rate at least $\min\{\rho_a,1\}$, independently of the input value $a^*$.
\end{theorem}

\begin{proof}
By Theorem~\ref{thm:exponential}, the two exponential modules converge to $e^{a_p^*}$ and $e^{a_n^*}$ respectively at rate at least $\min\{\rho_a, 1\}$. 
 The time-dependent-input inversion result of \cite{anderson2025arithmetic} gives that the reciprocal module converges to $1/e^{a_n^*}$ at the same rate. 
The multiplication module from \cite{anderson2025arithmetic}  has the analogous time-dependent-input rate guarantee and yields convergence of the product to $e^{a_p^*}/e^{a_n^*} = e^{a^*}$ at rate at least $\min\{\rho_a, 1\}$. 
\end{proof}
The explicit constructions of the reciprocal and multiplication modules are given in \cite{anderson2025arithmetic}.

\section{Computing the logarithm function at input-independent speed}
\label{sec:log}

We now turn to the logarithm, whose computation is substantially more subtle than that of the exponential. One basic difficulty is that $\ln a^*$ may be negative, whereas chemical concentrations are nonnegative, so any chemistry-based realization on the full domain $a^*>0$ must ultimately use the dual-rail representation introduced in Section~\ref{sec:framework}.  A second difficulty is that natural candidate systems either fail to satisfy mass-action kinetics, restrict the allowable input domain, or lose the property of input-independent speed. For this reason, we proceed iteratively: we begin with simple polynomial systems that clarify the underlying mechanism, then modify them step by step to recover chemical realizability, mass-action structure, full domain, and finally input-independent speed. The intermediate modules that arise along the way are also of independent interest. Indeed, in settings where one does not require all of the constraints imposed here simultaneously, these alternative constructions may themselves provide useful ways to compute the logarithm.

\subsection{A first attempt}

We begin with a simple polynomial dynamical system that computes $\ln(a)$ for $a>0$. Although this first construction is not mass-action, it captures the basic mechanism underlying the later modules and is motivated by the same idea used for the exponential.  Consider
\begin{align} \tag{System 1} \label{eq:polyforlog}
    \dot{x}(t)=a(t)-z,\qquad \dot{z}(t)=z(a(t)-z).
\end{align}
We first analyze this system in the simple case where $a(t)\equiv a$ is constant. The equation for $z$ is logistic, with carrying capacity $a$ and intrinsic growth rate $a$, and its solution for initial value $z(0)>0$ is
\[
z(t)=\frac{a}{1+\left(\frac{a}{z(0)}-1\right)e^{-at}},
\]
as may be verified by direct substitution. In particular, $z(t)\to a$ at rate $a$ as $t\to\infty$. Moreover,
\[
\frac{\dot z(t)}{z(t)}=a-z(t)=\dot x(t),
\]
and hence
\[
\ln z(t)-\ln z(0)=x(t)-x(0).
\]
Therefore, if the initial values satisfy $x(0)=\ln z(0)$, then
\[
x(t)=\ln z(t)
\qquad\text{for all } t\ge 0.
\]
Since $z(t)\to a$, it follows immediately that $x(t)\to \ln a$.

This system correctly captures the logarithmic relation, but it still falls short of the goals set out in the introduction in two important respects:
\begin{enumerate}[(i)]
    \item It is not mass-action, owing to the $-z$ term on the right-hand side of $\dot x(t)$. This may be acceptable in a non-chemical implementation, such as a polynomial dynamical system built from quantities that are allowed to take both positive and negative values.
    \item Its speed of convergence depends on the input $a$: as the following lemma shows, the rate of convergence is exactly $a$.  
\end{enumerate}

\begin{lemma}
\label{lem:poly_log_speed}
    The polynomial system \eqref{eq:polyforlog} with $a(t)\equiv a$ and $x(0)=\ln z(0)$ and $z(0)\ne a$ converges to $(z^*,x^*)=(a,\ln a)$ at rate exactly $a$; if $z(0)=a$ (and hence $x(0)=\ln a$), the system starts at equilibrium.
\end{lemma}

\begin{proof}
The explicit solution
\[
z(t)=\frac{a}{1+(a/z(0)-1)e^{-at}}
\]
gives
\[
z(t)-a = \frac{-a\bigl(1-a/z(0)\bigr)e^{-at}}{1+\bigl(a/z(0)-1\bigr)e^{-at}}.
\]
As $t\to\infty$ the denominator tends to $1$, so for all sufficiently large $t$,
\[
c_1 e^{-at} \le |z(t)-a| \le c_2 e^{-at}
\]
for positive constants $c_1,c_2$ depending on $a$ and $z(0)$ (assuming $z(0)\ne a$). This gives $z(t)\to a$ at rate exactly $a$. Since $x(t)=\ln z(t)$ and $\ln$ is $C^1$ near $a$ with nonzero derivative $1/a$, the same two-sided asymptotic holds for $x(t)-\ln a$, so $x(t)\to\ln a$ at rate exactly $a$.
\end{proof}

In the more general case where $a(t)$ converges to a positive value $a^*$ at rate at least $\rho_a$, there is still no uniform lower bound on the speed of computation. Indeed, applying Lemma~\ref{thm:system2_analysis} with $g_1(t)=a(t)$, $g_2(t)\equiv 1$, and $m=1$, we find that the rate of convergence is at least $\min\{\rho_a,a^*\}$. Thus, as the input $a^*$ gets smaller, the guaranteed speed of computation also becomes smaller. Equivalently, the time required to compute $\ln a^*$ to any fixed accuracy cannot be bounded uniformly over all $a^*>0$.

\subsection{A restricted-domain variant}

As noted in Section~\ref{sec:framework}, a necessary condition for a polynomial system to arise from mass-action kinetics is that any negative term in $\dot x_i$ must contain a factor of $x_i$. Thus \eqref{eq:polyforlog} is not mass-action, because of the  $-z$ term in $\dot x$.

Nevertheless, it is still natural to ask whether \eqref{eq:polyforlog} might be realizable within a broader chemistry-based framework. In particular, when $a>1$, the relation $x(t)=\ln z(t)$ shows that $x(t)$ remains nonnegative for all $t\ge 0$. 
Indeed, $z(t)\ge 1$ is maintained as an invariant: if $z(t)=1$ and $a(t)\ge 1$, then $\dot z(t)=z(t)(a(t)-z(t))=a(t)-1\ge 0$, so the trajectory cannot cross below~1. With $z(0)\ge 1$ and $a(t)\ge 1$ for all $t$, invariance follows by a standard comparison argument.
This suggests that, under the input restriction $a\ge 1$, the system may still admit a chemically meaningful interpretation, even though it is not mass-action in the strict sense. We refer to \eqref{eq:polyforlog} together with the restriction $a\ge 1$ as (System~1r). 

This restricted system also has a uniformly positive lower bound on its speed of convergence. Indeed, since $a^*\ge 1$, the bound $\min\{\rho_a,a^*\}$ from the previous subsection immediately yields
\[
\min\{\rho_a,a^*\}\ge \min\{\rho_a,1\}.
\]
Thus, on the restricted domain $a^*\ge 1$, the computation occurs at input-independent speed.  We record this as a corollary below.

\begin{corollary}[System~1r at input-independent speed]
\label{cor:system1r}
System~1r with $a(t)\ge 1$ for all $t\ge 0$, $a(t)\to a^*$ at rate at least $\rho_a$, $z(0)\ge 1$, and $x(0)=\ln z(0)$ computes $\ln a^*$ for all $a^*\ge 1$ at rate at least $\min\{\rho_a,1\}$.
\end{corollary}

\subsection{A full-domain non-mass-action variant}

The restricted-domain system (System~1r) achieves input-independent speed, but only for inputs satisfying $a^*\ge 1$. We now show that this domain restriction can be removed while preserving  a uniform lower bound on the convergence rate. The resulting construction is still not mass-action, but it provides a simple polynomial system that computes $\ln a^*$ for all $a^*>0$ at input-independent speed.
The key mechanism is borrowed from \cite{anderson2025arithmetic}: the logistic-type ODE $\dot y = y(1-a(t)y)$ computes the reciprocal $1/a^*$ at input-independent speed, because the extra factor of $y$ ensures the convergence rate has a uniform positive lower bound independent of the input. This device was used repeatedly in \cite{anderson2025arithmetic} to build arithmetic modules; here we use it as the foundation for an input-independent logarithm module.

\begin{lemma}
\label{lem:system2_log}
   Consider the following non-mass-action system, where $a(t)$ converges to a positive constant $a^*$ at rate at least $\rho_a$:
\begin{equation} \tag{System 2} \label{eq:polyforlog_ii}
    \begin{aligned}
    \dot y = y\left(1-a(t) y\right),  \qquad \dot x = 1-yz, \qquad \dot z = z(1-yz), 
\end{aligned}
\end{equation}
with initial values satisfying $y(0) > 0$, $z(0) > 0$, and $x(0) = \ln z(0)$. 
Then $x(t) \to \ln a^*$ at a rate of at least $\min\{\rho_{a},1\}$. 
\end{lemma}

\begin{proof}
Applying Lemma~\ref{thm:system2_analysis} to $\dot y = y(1-a(t)y),$ with $g_1(t)\equiv 1$, $g_2(t)=a(t)$, and $m=1$, we obtain $y(t)\to \frac{1}{a^*}$
at rate at least $\min\{1,\rho_a\}$.

Applying the same lemma to $\dot z = z(1-y(t)z),$ with $g_1(t)\equiv 1$, $g_2(t)=y(t)$, and $m=1$, we obtain
$z(t)\to \frac{1}{y^*}=a^*$
at rate at least $\min\{1,\rho_a\}$.

Since
\[
\dot x = 1-yz = \frac{\dot z}{z},
\]
integration gives
\[
x(t)-x(0)=\ln z(t)-\ln z(0).
\]
Thus, if $x(0)=\ln z(0)$, then
\[
x(t)=\ln z(t)
\qquad\text{for all } t\ge 0.
\]
Finally, since $\ln$ is locally Lipschitz near $a^*>0$, Lemma~\ref{lem:lipschitz} implies that
\[
x(t)=\ln z(t)\to \ln a^*
\]
at rate at least $\min\{1,\rho_a\}$.
\end{proof}

\subsection{A first mass-action variant}

We now modify the previous construction so that it satisfies mass-action kinetics. The idea is simple: in System~1, the obstruction to mass-action was the missing factor of $x$ in the negative term of $\dot x$. Multiplying the vector field by this missing factor yields the following system:
\begin{align} \tag{System 3}  \label{eq:massactionforlog}
    \dot{z}(t)=(a(t)-z)xz, \qquad
    \dot{x}(t)=(a(t)-z)x.
\end{align}
The corresponding reaction network is
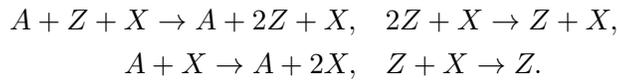
\begin{align*}
    A+Z+X\to A+2Z+X,\quad &2Z+X\to Z+X,\\
    A+X\to A+2X,\quad &Z+X\to Z.
\end{align*}
As before, we impose the compatibility condition $x(0)=\ln z(0)$.

Suppose first that $a(t)\to a^*\ge 1$. Since
\[
\dot x(t)=\frac{\dot z(t)}{z(t)}=\frac{d}{dt}\ln z(t),
\]
integration gives
\[
x(t)-x(0)=\ln z(t)-\ln z(0).
\]
Thus, if $x(0)=\ln z(0)$, then
\[
x(t)=\ln z(t)
\qquad\text{for all } t\ge 0.
\]
In particular, if $z(t)\to a^*$, then $x(t)\to \ln a^*$.

Now consider the case $a^*<1$, assuming for concreteness that $z(0)\ge 1$, so that $x(0)=\ln z(0)\ge 0$. As long as $x(t)>0$, the identity $x(t)=\ln z(t)$ remains valid, and the dynamics drive $z(t)$ downward toward the level $z=1$, where $x=\ln 1=0$. At that point both right-hand sides vanish, so the trajectory becomes trapped at $(z,x)=(1,0)$. Thus, for inputs $a^*\in(0,1)$, the system does not compute $\ln a^*$; instead, it computes the rectified logarithm
\[
\max\{\ln a^*,0\}=0.
\]
Accordingly, System~3 is restricted to the domain $a^*\ge 1$ if one wants the correct logarithmic output.

Even on the restricted domain $a^*>1$, the convergence speed still depends on the input and therefore has no uniform positive lower bound. This is already visible in the constant-input case.

\begin{lemma}
\label{lem:massaction_log_const}
The mass-action system \eqref{eq:massactionforlog} with $a(t)\equiv a>1$, $z(0)\ge a$, and $x(0)=\ln z(0)$ converges to $(z^*,x^*)=(a,\ln a)$ at rate  exactly $a\ln a$ provided $z(0)\ne a$; if $z(0)=a$, the trajectory is already at equilibrium.
\end{lemma}
\begin{proof}
Since
\[
\dot x=\frac{\dot z}{z},
\]
the condition $x(0)=\ln z(0)$ implies
\[
x(t)=\ln z(t)
\qquad\text{for all }t\ge 0.
\]
Thus it suffices to study the scalar equation
\[
\dot z=(a-z)z\ln z.
\]
Because $z(0)\ge a>1$, the trajectory satisfies $z(t)\ge a$ for all $t\ge 0$, and hence
\[
z(t)\ln z(t)\ge a\ln a.
\]
Therefore
\[
\dot z = -(z-a)\,z\ln z \le -(a\ln a)(z-a).
\]
By Gronwall's inequality,
\[
0\le z(t)-a \le (z(0)-a)e^{-a\ln (a)\,t}.
\]
Since $\ln$ is locally Lipschitz near $a>1$, Lemma~\ref{lem:lipschitz} then yields
\[
x(t)=\ln z(t)\to \ln a
\]
at rate at least $a\ln (a)$.
To see that this rate is exact when $z(0)\ne a$, set $u(t)=z(t)-a$. Then
\[
\frac{d}{dt}\ln|u(t)|=-z(t)\ln z(t),
\]
so, since $z(t)\to a$,
\[
-\lim_{t\to\infty}\frac{1}{t}\ln|u(t)|
=\lim_{t\to\infty}\frac1t\int_0^t z(s)\ln z(s)\,ds
=a\ln a.
\]
By the mean value theorem, $x(t)-\ln a=(z(t)-a)/\xi(t)$ for some $\xi(t)$ between $z(t)$ and $a$. Since $\xi(t)\to a>0$,
\[
-\frac{1}{t}\ln|x(t)-\ln a|
=-\frac{1}{t}\ln|z(t)-a|+\frac{1}{t}\ln\xi(t)
\longrightarrow a\ln a,
\]
so $x(t)\to\ln a$ at the same asymptotic rate.
\end{proof}

Since $a\ln a \to 0$ as $a \to 1^+$,  the exact constant-input rate tends to zero, and hence there is no positive lower bound on the convergence speed that is uniform over all $a>1$.

We next consider the case of time-varying input. The corresponding rate estimate will be used later in the analysis of Systems~5 and~6.

\begin{lemma}[Speed of System~3 under a uniform lower bound]
\label{lem:massaction_log_speed}
Consider the mass-action system \eqref{eq:massactionforlog}, and assume that $a(t)\to a^*>1$ at rate at least $\rho_a>0$. Suppose there exists a constant $c>1$ such that
\[
a(t)\ge c \qquad \text{for all } t\ge 0,
\]
and assume the initial conditions satisfy
\[
z(0)\ge c,
\qquad
x(0)=\ln z(0)>0.
\]
Then $(z(t),x(t))\to (a^*,\ln a^*)$ at rate at least
\[
\min\{\rho_a,\; c\ln c\}.
\]
\end{lemma}
\begin{proof}
As established above, the condition $x(0)=\ln z(0)$ implies
\[
x(t)=\ln z(t)
\qquad \text{for all } t\ge 0.
\]
Hence it suffices to study the scalar equation
\[
\dot z=(a(t)-z)\,z\ln z.
\]

Set
\[
u(t)=z(t)-a^*,
\qquad
\varepsilon(t)=a(t)-a^*.
\]
 Fix any $\eta\in(0,\rho_a)$. By Definition~\ref{def:speed}, there exist $C_\eta>0$ and $T_\eta\ge0$ such that
$|\varepsilon(t)|\le C_\eta e^{-(\rho_a-\eta)t}$ for all $t\ge T_\eta$. Then
\[
\dot u
=
(a(t)-z)\,z\ln z
=
(\varepsilon(t)-u(t))\,z(t)\ln z(t).
\]

We first note that $z(t)\ge c$ for all $t\ge 0$. Indeed, if at some time $z(t)=c$, then
\[
\dot z(t)=(a(t)-c)\,c\ln c \ge 0,
\]
since $a(t)\ge c$. Thus the trajectory cannot cross below $c$.

Next, since $a(t)\to a^*$, the function $a(t)$ is bounded, and therefore $z(t)$ is also bounded above: if $z(t)$ exceeds $\max\{z(0),\sup_{s\ge0}a(s)\}$, then $\dot z(t)<0$. Hence there exists $K<\infty$ such that
\[
c\le z(t)\le K
\qquad \text{for all } t\ge 0.
\]
It follows that
\[
c\ln c \le z(t)\ln z(t)\le K\ln K
\qquad \text{for all } t\ge 0.
\]

Therefore, writing
\[
m:=c\ln c,
\qquad
M:=K\ln K,
\]
we obtain
\begin{align*}
\frac{d}{dt}|u(t)|
&= \sgn(u(t))\,\dot u(t)
\quad\text{(since $\tfrac{d}{dt}|u|=\sgn(u)\dot u$ a.e.\ for absolutely continuous $u$)}\\
&= \sgn(u(t))\,(\varepsilon(t)-u(t))\,z(t)\ln z(t)\\
&\le -m\,|u(t)| + M\,|\varepsilon(t)|\\
&\le -m\,|u(t)| + MC_\eta\,e^{-(\rho_a-\eta)t},\qquad t\ge T_\eta.
\end{align*}

 A standard Gronwall/comparison argument applied from $T_\eta$  shows that $u(t)$ converges to zero at rate at least $\min\{\rho_a-\eta,m\}$. Since $\eta\in(0,\rho_a)$ is arbitrary, the asymptotic rate is at least $\min\{\rho_a,m\}$. Hence
\[
z(t)\to a^*
\]
at rate at least $\min\{\rho_a,c\ln c\}$.

Finally, since $x(t)=\ln z(t)$ and $\ln$ is locally Lipschitz near $a^*>1$, Lemma~\ref{lem:lipschitz} implies that
\[
x(t)\to \ln a^*
\]
at rate at least $\min\{\rho_a,c\ln c\}$.
\end{proof}

\begin{corollary}[System~3 with an eventual lower bound]
\label{lem:massaction_log_eventual}
Suppose the input to System~3 satisfies $a(t)\to a^*\ge e$ at rate at least $\rho_a>0$, and assume $z(0)>1$ and $x(0)=\ln z(0)$. Then
$(z(t),x(t))\to(a^*,\ln a^*)$ at rate at least $\min\{\rho_a,1\}$.
\end{corollary}
\begin{proof}
Since $a^*\ge e>2$, there exist $T_0$ and $\eta>0$ such that
$a(t)\ge2+\eta$ for all $t\ge T_0$. The invariant relation
$x(t)=\ln z(t)$ and $z(0)>1$ imply $z(t)>1$ for all $t$. As long as
$z(t)<2$ for $t\ge T_0$,
\[
\dot z(t)\ge \eta z(t)\ln z(t)
\ge \eta z(T_0)\ln z(T_0)>0,
\]
so $z$ reaches $2$ in finite time. Thus, for some $T_1\ge T_0$,
$z(T_1)\ge2$. Applying Lemma~\ref{lem:massaction_log_speed} from
$T_1$ onward with $c=2$ gives rate at least
\[
\min\{\rho_a,2\ln2\}\ge\min\{\rho_a,1\}.
\]
\end{proof}

\subsection{A dual-rail mass-action system on the full positive domain}

We now remove the domain restriction of System~3 while remaining within mass-action kinetics. 
We  use the dual-rail representation introduced in Section~\ref{sec:framework}. Thus, instead of a single output species $X$, we use a pair $(X_p,X_n)$, with intended output
\[
x \coloneqq x_p-x_n=\ln a^*.
\]
In this encoding, $x_p$ records the positive part of the logarithm, while $x_n$ records the magnitude of the negative part.

The positive rail is obtained directly from System~3:
\begin{align} \tag{System 4p} \label{eq:massactionforlog_pos}
    \dot{z}(t)=(a(t)-z)x_pz, \qquad
    \dot{x}_p(t)=(a(t)-z)x_p.
\end{align}
Assuming the compatibility condition $x_p(0)=\ln z(0)$, we have
\[
x_p(t)=\ln z(t)
\qquad\text{for as long as }x_p(t)>0.
\]
Thus, if $a^*\ge 1$, then $z(t)\to a^*$ and hence
\[
x_p(t)\to \ln a^*.
\]
If instead $a^*<1$, then the trajectory is driven to the boundary value $z=1$, yielding
\[
x_p(t)\to 0.
\]
Accordingly, the positive rail computes the positive part of the logarithm.

To obtain the negative rail, we apply the same mechanism to the reciprocal target $1/a^*$. Specifically, consider
\begin{align} \tag{System 4n} \label{eq:massactionforlog_neg}
    \dot{y}(t)=(1-a(t)y)x_n y, \qquad
    \dot{x}_n(t)=(1-a(t)y)x_n.
\end{align}
We assume throughout that $y(0)>1$ and $x_n(0)=\ln y(0)$. Since
\[
\frac{\dot y(t)}{y(t)}=(1-a(t)y)x_n=\dot x_n(t),
\]
it follows that
\[
x_n(t)=\ln y(t)
\qquad\text{for all } t\ge 0.
\]
Thus the trajectory remains on the invariant curve $x_n=\ln y$, and the dynamics reduce to
\[
\dot y=(1-a(t)y)y\ln y.
\]

This is the analog of System~3 applied to the reciprocal target $1/a(t)$. Indeed, writing $b(t)=1/a(t)$, we have
\[
\dot y=a(t)\,(b(t)-y)\,y\ln y,
\]
which differs from the reduced form of System~3 only by the positive factor $a(t)$ and therefore has the same qualitative behavior. Consequently, if $a^*<1$, then $1/a^*>1$, so $y(t)\to 1/a^*$ and hence
\[
x_n(t)=\ln y(t)\to \ln(1/a^*)=-\ln a^*.
\]
If instead $a^*\ge 1$, then $1/a^*\le 1$, and since $y(t)\ge 1$ along the invariant curve, the trajectory is driven to the boundary value $y=1$, yielding $x_n(t)\to 0$.

For constant input $a\in(0,1)$ and a non-equilibrium compatible initialization, the convergence rate of the negative rail is exactly
\[
-\ln a=\ln(1/a).
\]
Indeed, with $y^*=1/a$ and $u(t)=y(t)-y^*$,
\[
\frac{d}{dt}\ln|u(t)|=-a\,y(t)\ln y(t),
\]
and therefore the asymptotic rate is $a y^*\ln y^*=\ln(1/a)$. Thus, if the reciprocal target satisfies $1/a\ge c>1$, the constant-input rate is bounded below by $\ln c$.

Combining \eqref{eq:massactionforlog_pos} with \eqref{eq:massactionforlog_neg} yields the dual-rail logarithm:
\begin{itemize}
    \item When $a^* \ge 1$: $x_p \to \ln a^*$ and $x_n \to 0$, so $x_p - x_n \to \ln a^*$.
    \item When $a^* < 1$: $x_p \to 0$ and $x_n \to -\ln a^*$, so $x_p - x_n \to \ln a^*$.
\end{itemize}
We refer to this combined system as (System 4). The issue of input-dependent speed, however, persists. In the following subsections we show how composing these modules with the arithmetic modules of~\cite{anderson2025arithmetic} resolves this difficulty, yielding two further constructions---Systems~5 and~6---that achieve input-independent speed.

\subsection{Composition of logarithm modules with arithmetic operations}
\label{sec:composite}

The fundamental mechanism behind  the composition results below is the  time-dependent-input rate guarantee for each component module: when its inputs are themselves converging, the module's output converges at a rate bounded below by the minimum of the input rates and the module's intrinsic speed. Crucially, this holds even in \emph{series} composition: the downstream module begins running from $t=0$ with a time-varying input $u(t)\to u^*$, without waiting for $u(t)$ to converge. For example, to compute $\ln(a+b)$, the addition module produces $z(t)\to a^*+b^*$ while the logarithm module simultaneously receives $z(t)$ as input and produces $\ln z(t)\to\ln(a^*+b^*)$---both running from $t=0$. As a result,  the convergence rate of the composite computation is controlled only by the external-input rates and the minimum intrinsic speed among the component modules (Theorem~\ref{thm:composite_log}).

We also recall from Remark~\ref{rem:init} that the exponential and logarithm modules require specific initialization conditions; in any composition, these must be satisfied for each component module independently.

By a \emph{composite computation} we mean a finite feedforward directed acyclic graph
whose non-source nodes are modules and whose source nodes receive the external inputs of the computation.
We refer to such a graph as the \emph{computational graph}.
The following theorem shows that input-independent speed is preserved by any computational graph.

\begin{theorem}[Composition of logarithm and arithmetic modules]
\label{thm:composite_log}
Consider a computation composed from a finite number of modules drawn from the following:
\begin{enumerate}
    \item the arithmetic modules of~\cite{anderson2025arithmetic} (identification, inversion, $m$th roots, addition, multiplication, absolute difference, rectified subtraction);
    \item the nonnegative exponential module (Theorem~\ref{thm:exponential}) and the real exponential module (Theorem~\ref{thm:real_exponential});
    \item any of the logarithm modules of Systems~1--4 (Lemma~\ref{lem:poly_log_speed}, Corollary~\ref{cor:system1r}, Lemma~\ref{lem:system2_log}, and the constructions of Section~\ref{sec:log}).
\end{enumerate}
Suppose that each component module has a positive intrinsic speed, all rate constants equal~$1$, the external inputs to the composite computation converge at rate at least $\rho>0$, and all admissible initialization conditions are satisfied. Then the output of the composite computation converges to the correct value at rate at least $\min\{\rho,\rho_{\min}\}$, where $\rho_{\min}$ is the minimum intrinsic speed  among the component modules.
\end{theorem}

\begin{proof}

Let $\rho$ be the least of the convergence rates of the external inputs, let
$\rho_{\min}=\min_M\rho_M$ over the modules $M$ of the computation, and set
$r_0:=\min\{\rho,\rho_{\min}\}$. We show, by induction on the number of modules $N$, that
every module output converges at rate at least $r_0$.

If $N=1$, all inputs of the single module are external and converge at rate at least
$\rho$. Since the module has intrinsic speed $\rho_M$, its output converges at rate at least
$\min\{\rho,\rho_M\}\ge r_0$.

Let $N>1$, and let $M$ be a sink of the computational graph, that is, a module whose output is not an
input of any other module. Deleting $M$ leaves a computation with $N-1$ modules, whose
external inputs converge at rate at least $\rho$ and whose modules have intrinsic speed at
least $\rho_{\min}$; by the induction hypothesis, each of their outputs converges at rate at
least $r_0$. Every input of $M$ is therefore either external, converging at rate at least
$\rho\ge r_0$, or the output of one of these modules, converging at rate at least $r_0$.
Since $M$ has intrinsic speed $\rho_M\ge\rho_{\min}\ge r_0$, its output converges at rate at
least $\min\{r_0,\rho_M\}=r_0$.

In particular, the output of the composite computation converges at rate at least
$r_0$.
\end{proof}

\begin{remark}
 The following special cases illustrate typical values of $\rho_{\min}$:
\begin{itemize}
    \item if all component modules are  exponential (Theorems~\ref{thm:exponential},~\ref{thm:real_exponential}) or arithmetic modules having intrinsic speed~$1$, then $\rho_{\min}=1$;
    \item if all logarithm modules are restricted to System~1r ($a^*\ge 1$) or System~2 (full domain), then $\rho_{\min}=1$;
\end{itemize}
\end{remark}

\subsection[A mass-action system with input-independent speed on the restricted domain]{A mass-action system with  input-independent speed on the restricted domain}

We now apply Theorem~\ref{thm:composite_log} to obtain the first mass-action logarithm construction with input-independent speed. The key insight is to shift the input away from the regime near $a^*=1$, where the convergence rate of System~3 becomes arbitrarily slow. Multiplying the input by $e>1$ ensures the shifted input  converges to a value at least $e$, so Corollary~\ref{lem:massaction_log_eventual} gives a uniform lower bound on the speed, and a final subtraction recovers the correct logarithm value.

Assume that the input is $a(t)$ with limiting value $a^*\ge 1$. We use the following three steps.  Below, the multiplication and subtraction steps are implemented using the arithmetic modules from \cite{anderson2025arithmetic}, and the preservation of input-independent speed under composition follows from Theorem~\ref{thm:composite_log}.
\begin{enumerate}
    \item Multiply by $e$\footnotemark:
    \[
    a^* \mapsto ea^*.
    \]
    \item Apply the logarithm module from \eqref{eq:massactionforlog}:
    \[
    ea^* \mapsto \ln(ea^*) = 1+\ln a^*.
    \]
    
    \item Subtract $1$:
    \[
    1+\ln a^* \mapsto \ln a^*.
    \]
\end{enumerate}

The final output is $\ln a^*$. We refer to this composite construction as System~5.

\begin{theorem}[System~5 at input-independent speed]
\label{thm:system5}
For any input $a(t)\to a^*\ge 1$ at rate at least $\rho_a>0$, System~5 computes $\ln a^*$ at rate at least $\min\{\rho_a,\,1\}$.
\end{theorem}
\begin{proof}
Step~1 (multiplication by $e$) and Step~3 (subtraction of~$1$) are arithmetic modules from \cite{anderson2025arithmetic} with input-independent speed lower bound~$1$. For Step~2,  the output of Step~1 converges to $ea^*\ge e$ at rate at least $\min\{\rho_a,1\}$. Corollary~\ref{lem:massaction_log_eventual} therefore applies to the System~3 logarithm stage and gives convergence rate at least $\min\{\rho_a,1\}$. The overall rate of at least $\min\{\rho_a,\,1\}$ then follows from Theorem~\ref{thm:composite_log}.
\end{proof}

\footnotetext{Any constant $c > 1$ can replace $e$, changing Step 2's output to $\ln(ca^*) = \ln c + \ln a^*$ and Step 3's subtraction to $\ln c$ rather than $1$. The value $e$ is convenient because the subtraction in Step 3 is exactly $1$. Note also that the constant $e$ itself can be produced by running the exponential module \eqref{eq:exponential} with input $a^* = 1$, making this construction self-contained.}

\begin{remark}[Solution length]
\label{rem:length}
The convergence-rate notion used here is distinct from the \emph{length} of the solution curve used as a complexity measure by Bournez, Gra\c{c}a, and Pouly \cite{bournez2017ptime}; solution length is not analyzed here.
\end{remark}

\subsection[A mass-action system with input-independent speed on the full positive domain]{A mass-action system with  input-independent speed on the full positive domain}

We now arrive at the final construction of the paper, which resolves the logarithm problem in full. Specifically, we construct a mass-action system that computes $\ln a^*$ for every input $a^*>0$, does so in dual-rail representation, and achieves input-independent speed. The idea is to combine the logarithm module developed above with the arithmetic modules from \cite{anderson2025arithmetic} in such a way that every intermediate computation occurs at uniformly lower-bounded speed. The reciprocal, multiplication, subtraction, and maximum steps are implemented using modules from \cite{anderson2025arithmetic}, while the logarithm step is supplied by the construction developed in the previous subsection. By Theorem~\ref{thm:composite_log}, the overall construction also computes at uniformly lower-bounded speed.

We assume that the input is $a(t)$ with limiting value $a^*>0$. The construction proceeds as follows.

\begin{enumerate}
    \item Compute the reciprocal and record it in a separate species:
    \[
    a^* \mapsto \left(a^*, \frac{1}{a^*}\right).
    \]

    \item Multiply each component by $e$, the base of the natural logarithm:
    \[
    \left(a^*, \frac{1}{a^*}\right) \mapsto \left(ea^*, \frac{e}{a^*}\right).
    \]

    \item Compute the maximum of each component with $e$:
    \[
    \left(ea^*, \frac{e}{a^*}\right) \mapsto \left(\max\{ea^*,e\}, \max\left\{\frac{e}{a^*},e\right\}\right)
    =
    \begin{cases}
        (ea^*, e) & \mbox{if } a^* \ge 1, \\[0.3em]
        \left(e, \frac{e}{a^*}\right) & \mbox{if } a^* < 1.
    \end{cases}
    \]

    \item Apply the logarithm module to each component using the construction in \eqref{eq:massactionforlog}:
    \[
    \begin{cases}
        (ea^*, e) & \mbox{if } a^* \ge 1, \\[0.3em]
        \left(e, \frac{e}{a^*}\right) & \mbox{if } a^* < 1
    \end{cases}
    \mapsto
    \begin{cases}
        (1+\ln a^*, 1) & \mbox{if } a^* \ge 1, \\[0.3em]
        \left(1, 1-\ln a^*\right) & \mbox{if } a^* < 1.
    \end{cases}
    \]

    \item Subtract $1$ from each component:
    \[
    \begin{cases}
        (1+\ln a^*, 1) & \mbox{if } a^* \ge 1, \\[0.3em]
        \left(1, 1-\ln a^*\right) & \mbox{if } a^* < 1
    \end{cases}
    \mapsto
    \begin{cases}
        (\ln a^*, 0) & \mbox{if } a^* \ge 1, \\[0.3em]
        \left(0, -\ln a^*\right) & \mbox{if } a^* < 1.
    \end{cases}
    \]
\end{enumerate}

The final output is $\ln a^*$ in dual-rail representation. We refer to this composite construction as System~6.

\begin{theorem}[System~6 at input-independent speed]
\label{thm:system6}
For any input $a(t)\to a^*>0$ at rate at least $\rho_a>0$, System~6 computes $\ln a^*$ in dual-rail representation at rate at least $\min\{\rho_a,\,1\}$.
\end{theorem}
\begin{proof}
Steps~1 and~2 (reciprocal and multiplication by $e$) and Step~5 (subtraction of~$1$) are arithmetic modules from \cite{anderson2025arithmetic}. Step~3 (maximum with $e$) is also from \cite{anderson2025arithmetic}. In Step~4, each component  converges to a limiting value at least $e$ at rate at least $\min\{\rho_a,1\}$, so Corollary~\ref{lem:massaction_log_eventual} applies separately to each System~3 logarithm rail and gives rate at least $\min\{\rho_a,1\}$ for that step. The full rate then follows from Theorem~\ref{thm:composite_log}.
\end{proof}

\section{Composite computations at input-independent speed}
\label{sec:general_composite}

The following theorem extends the composition result of Theorem~\ref{thm:composite_log} to cover all modules in this paper, including Systems~5 and~6.

\begin{theorem}[General composite computations at input-independent speed]
\label{thm:composite}
Consider a computation composed from a finite number of modules drawn from the arithmetic modules of~\cite{anderson2025arithmetic}, the exponential modules of Theorems~\ref{thm:exponential} and~\ref{thm:real_exponential}, and any logarithm module of this paper (Systems~1--6). Suppose that each component module has a positive intrinsic speed, all rate constants equal~$1$, the external inputs to the composite computation converge at rate at least $\rho>0$, and all initialization conditions are satisfied. Then the output converges to the correct value at rate at least $\min\{\rho,\rho_{\min}\}$, where $\rho_{\min}$ is the minimum intrinsic speed  among the component modules.
\end{theorem}

\begin{proof}
Identical to that of Theorem~\ref{thm:composite_log}. Systems~5 and~6 have intrinsic speed $\rho_M=1$ (Theorems~\ref{thm:system5} and~\ref{thm:system6}), so they may also be used as nodes of the computational graph.
\end{proof}

\begin{corollary}[Rate-$1$ lower bound]
\label{cor:composite_rate1}
Under the hypotheses of Theorem~\ref{thm:composite}, if  every component module  has intrinsic speed at least~$1$ and $\rho\ge 1$, then the composite computation proceeds at rate at least~$1$.
\end{corollary}

As a concrete illustration, consider computing $f(a,b)=e^a\cdot\ln b$ for $a\in\R$ and $b>0$. This composes the real exponential module (Theorem~\ref{thm:real_exponential}), the full-domain logarithm module (Theorem~\ref{thm:system6}), and a multiplication module from~\cite{anderson2025arithmetic}. If $a(t)\to a^*$ and $b(t)\to b^*$ each at rate at least $\rho$, with all initialization conditions satisfied, then by Theorem~\ref{thm:composite} the product $e^{a^*}\cdot\ln b^*$ is computed at rate at least $\min\{\rho,1\}$, independently of $a^*$ and $b^*$---and both computations run simultaneously from $t=0$.

\section{Sensitivity to initialization}\label{sec:init}

All of the constructions in this paper rely on an exact algebraic relation between an output species and an internal species---$x=e^z$ for the exponential and $x=\ln z$ (or the corresponding relation on a dual rail) for the logarithm---that is preserved because the initial conditions satisfy a compatibility condition. It is therefore natural to ask what happens when this condition is violated. The answer depends on the module. For some of the simpler systems the mismatch produces a fixed multiplicative or additive bias, whereas for the mass-action logarithm systems it can also change the attracting state by driving an output rail to the boundary at zero.

For the exponential module, the conserved quantity is
\[
\ln x(t)-z(t)=\ln x(0)-z(0).
\]
Thus, if $\ln x(0)-z(0)=\delta$, then
\[
x(t)=e^{\delta}e^{z(t)},
\]
and hence, whenever $z(t)\to a^*$, the output converges to $e^{\delta}e^{a^*}$. The initialization error is therefore a persistent multiplicative factor $e^{\delta}$. The same conclusion applies separately to the two exponential submodules used in the real-exponential construction.

For Systems~1 and~2, the internal $z$-dynamics do not depend on the output species $x$. Moreover,
\[
x(t)-\ln z(t)=x(0)-\ln z(0)=\delta.
\]
Consequently, if the internal state still converges to $z(t)\to a^*$, then
\[
x(t)\to\ln a^*+\delta.
\]
For these two systems the compatibility error is therefore an exact persistent additive bias.

The situation is different for the mass-action System~3, because the output species appears in the vector field for the internal state. If $x(0)-\ln z(0)=\delta$, then the same conserved quantity gives $x(t)=\ln z(t)+\delta$, and the reduced dynamics become
\[
\dot z=(a(t)-z)(\ln z+\delta)z.
\]
Thus the mismatch does more than translate the reported value: it introduces the boundary equilibrium $x=0$, $z=e^{-\delta}$. For constant input $a>0$, the intended equilibrium $z=a$ lies in the positive-output region precisely when $\ln a+\delta>0$. In that case the compatible-branch output is $\ln a+\delta$; if instead $\ln a+\delta\le0$, the nonnegative concentration $x$ cannot follow the formal additive shift below zero, and the boundary $x=0$ can become the attracting state. Hence an initialization mismatch can change the attractor, not merely add a constant error. At the threshold $\ln a+\delta=0$, the asymptotic convergence behavior is degenerate and need not retain the rate of the correctly initialized system.

Systems~4p and~4n inherit the same phenomenon rail by rail. For constant input, writing $\delta_p=x_p(0)-\ln z(0)$ and $\delta_n=x_n(0)-\ln y(0)$, the corresponding interior limits, when those rails remain away from zero, are $\ln a+\delta_p$ and $-\ln a+\delta_n$, respectively. If either quantity would be nonpositive, the associated mass-action rail can instead be trapped at its zero boundary. Thus the dual-rail output is generally piecewise in the initialization mismatch rather than a globally additive perturbation of $\ln a$.

Systems~5 and~6 deliberately shift the arguments supplied to System~3 away from the boundary before taking logarithms. This gives a positive margin for sufficiently small initialization errors in those logarithm stages. Nevertheless, the subsequent rectified-subtraction operations make the final perturbation piecewise, so there is no global statement that an arbitrary compatibility error passes unchanged through the entire composite construction.

This behavior stands in contrast to the arithmetic modules of \cite{anderson2025arithmetic}, which are self-correcting on their admissible strictly positive state space: their outputs converge to the correct value without a codimension-one compatibility relation between internal and output species. The distinction is not an artifact of our particular constructions. Hemery and Fages characterize the functions a mass-action CRN can \emph{stabilize}---drive to $f(X)$ from a full-dimensional domain of admissible initial conditions for the intermediate and output species, their notion of absolute functional robustness---as \emph{exactly} the real algebraic functions; in their notation $\mathcal{F}_S=\mathcal{F}_A$ \cite[Thm.~11]{hemery2024absolute}. The same algebraic boundary arises independently in the Lyapunov real-time framework of Fletcher, Klinge, Lathrop, Nye, and Rayman \cite{fletcher2025}, who prove that the reals CRN-computable in that sense are exactly the algebraic numbers.
Our argument uses only the \emph{necessity} direction $\mathcal{F}_S\subseteq\mathcal{F}_A$ \cite[Lem.~14]{hemery2024absolute}---that a stabilizable function must satisfy a nonzero polynomial relation---and not their converse compilation result. Since $\exp$ and $\ln$ are transcendental, they satisfy no such relation, so no mass-action CRN stabilizes them. Some degree of initialization sensitivity is therefore unavoidable, and the compatibility conditions imposed here are an intrinsic feature of directly computing a transcendental function, not a deficiency of the specific networks.

This settles, in the negative, a question one might ask about these constructions---whether some alternative could avoid initialization sensitivity altogether: none can be self-correcting on a full-dimensional domain of admissible initial conditions in the way the arithmetic modules are. The mechanism is exactly that of \cite[Ex.~6]{hemery2024absolute}: our modules attain the exact value only on the measure-zero manifold fixed by the compatibility relation (e.g. $x(0)=\ln z(0)$), whose empty interior is precisely what the stabilization notion excludes.

\section{Discussion}
\label{sec:discuss}

The constructions developed in this paper illustrate that there is no single canonical way to compute the logarithm by reaction networks. Rather, each construction represents a different trade-off among desirable properties, including chemical realizability, mass-action structure, full input domain, and input-independent speed. Figure~\ref{systems} summarizes these trade-offs. In particular, while System~6 is the most complete construction---it is mass-action, valid on the full positive domain, and has input-independent speed---it is also the most resource-intensive, since it is obtained by composing several modules. In settings where the number of species or the complexity of the implementation is the dominant constraint, one of the simpler constructions may be preferable, even at the cost of restricted domain or input-dependent speed.

More broadly, this work should be viewed as a beginning rather than an endpoint. The exponential and logarithm functions were chosen here as natural first test cases: they are foundational transcendental functions, they already exhibit many of the main obstacles that arise in chemistry-based analog computation, and they serve as useful templates for more complicated function evaluations. The results of this paper show that these two functions can be computed directly, without truncating power series, while still admitting rigorous convergence guarantees---exact output, arbitrary accuracy given sufficient time, and input-independent speed---over their entire domains and in carefully designed settings. Furthermore, the modules compose  in finite feedforward computations satisfying the individual module hypotheses: by Theorems~\ref{thm:composite_log} and~\ref{thm:composite}, any computation built from these modules and the arithmetic modules of~\cite{anderson2025arithmetic} that satisfies those hypotheses inherits the same guarantee,  with the convergence rate determined by the external-input rates and the intrinsic speeds of the component modules. A natural next step is to extend these ideas to other transcendental functions and to develop more systematic principles for constructing efficient reaction network modules beyond the exponential and logarithm.  At the same time, even for the exponential and logarithm themselves,
basic structural questions remain.

\begin{figure}[h!]
  \begin{center}

\begin{tikzpicture}[scale=1.5]
\node (chem) at (-4,0) {}; %% CHEM
\draw [ultra thick, draw=black, fill=yellow!20, opacity=0.3] (chem)  ellipse (3 and 3);
\node  [above = 3 of chem] {\bf Chemistry};
\node  [below = 2.5 of chem] {\specialcell{(System 1r)}};

\node (ma) at (-3.5,0) {}; %% MASS-ACTION
\draw [ultra thick, draw=black, fill=red!20, opacity=0.3] (ma)  ellipse (2.5 and 1.5);
\node  [above left = 0.4 of ma] {\specialcell[]{\bf Mass-action}};
\node  [left = 1 of ma] { \specialcell{\eqref{eq:massactionforlog}}};
    
\node (domain) at (0,0) {}; %% DOMAIN
\draw [ultra thick, draw=black, fill=blue!20, opacity=0.3] (domain)  ellipse (3 and 3);
\node  [above = 2 of domain] {\bf \specialcell{ Full\\Domain}};
\node  [right = 2 of domain] { \specialcell{\eqref{eq:polyforlog}}};
\node  [below = 2 of domain] { \specialcell{\eqref{eq:polyforlog_ii}}};

\node (bounded) at (-2,-1.75) {}; %% BOUNDED
\draw [ultra thick, draw=black, fill=green!10, opacity=0.3] (bounded)  ellipse (3.5 and 2);
\node  [below = 0.2 of bounded] {\bf \specialcell{\\\\\\Input-independent\\speed}};

\node  [above = 3 of bounded] { \specialcell{(System 
4)}};
\node  [below = 0.7 of chem] {\specialcell{(System 5)}};
\node  [above = 1.5 of bounded] { \specialcell{(System 6)}};

\end{tikzpicture}
\caption[Comparison of the logarithm constructions developed in this paper. Each ellipse represents a desirable property: Chemistry, Mass-action, Full Domain, and Input-independent speed.]{Comparison of the logarithm constructions developed in this paper. Each ellipse represents a desirable property: \textbf{Chemistry} (interpreted here in a broad sense, namely compatibility with a chemistry-inspired realization using nonnegative-valued kinetics, not necessarily strict mass-action), \textbf{Mass-action} (satisfies mass-action kinetics), \textbf{Full Domain} (computes $\ln a^*$ for all $a^*>0$), and \textbf{Input-independent speed} (converges at uniformly lower-bounded speed). Each system is placed according to the properties it satisfies.}
\label{systems}
  \end{center}
 \end{figure}
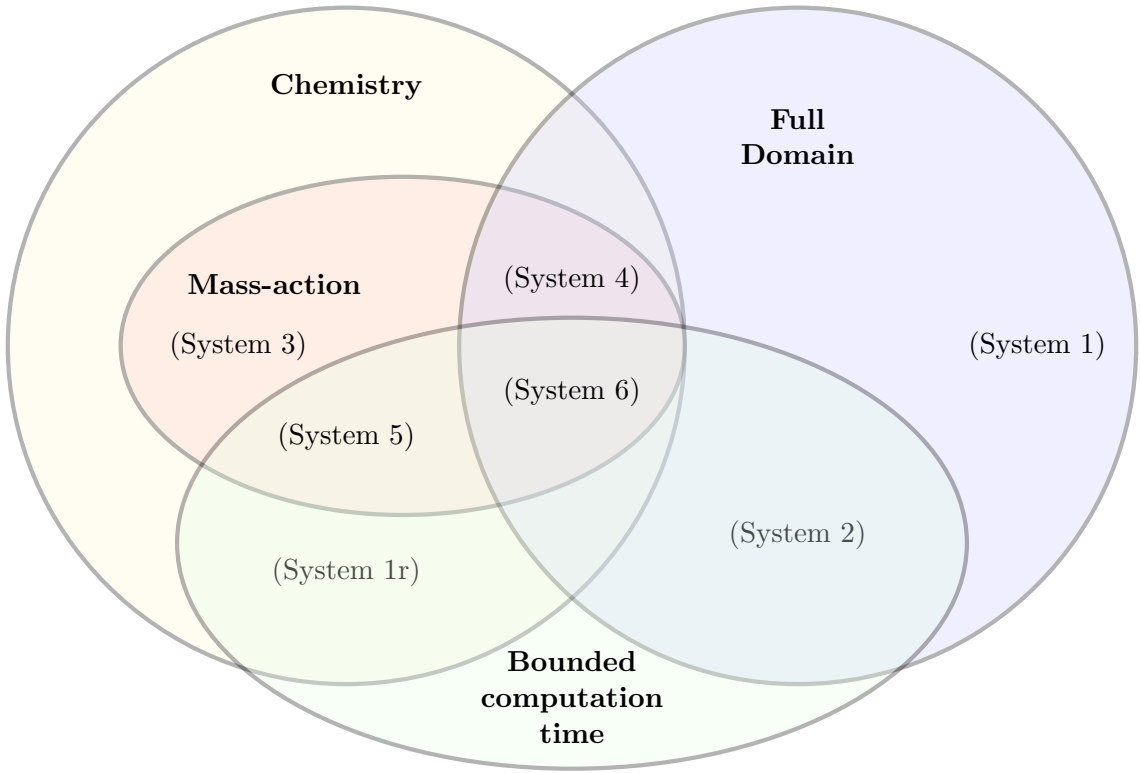

\appendix

\section{Background information}

The following two lemmas are from \cite{anderson2025arithmetic} (Lemmas 3.6 and 3.7 in the original work), restated here for easy reference. 

\begin{lemma}
\label{thm:system1_analysis}
Let $g_1:\R_{\ge 0} \to \R$ be a real-valued function that converges to $g_1^* \in \R$ at a rate of at least $\rho_{g_1}$; let $g_2:\R_{\ge 0} \to \R$ be a real-valued function that converges to a positive limit $g_2^* \in \R_{> 0}$ at a rate of at least $\rho_{g_2}$. 
We assume that $g_1, g_2$ are smooth enough so that for any $x(0) = x_0 \ge 0$, the following non-autonomous differential equation has a unique solution $x: \R_{\ge 0} \to \R$ for all time: 
\begin{align}\label{eq:linear_simple}
\dot x(t) &= g_1(t) - g_2(t)x(t). 
\end{align}
Then $x(t)$ converges to $g_1^*/g_2^*$ at a rate of at least
\begin{align}
\min\{\rho_{g_1}, \rho_{g_2}, g_2^*\}. 
\end{align}
\end{lemma}

\begin{lemma}
\label{thm:system2_analysis}
For $i \in \{1,2\}$, let $g_i:\R_{\ge 0} \to \R$ be a real-valued function that converges to a positive limit $g_i^* \in \R_{> 0}$ at a rate of at least $\rho_{g_i}$. 
We assume that $g_1, g_2$ are smooth enough so that for any $x(0) = x_0 > 0$, the following non-autonomous differential equation has a unique solution $x: \R_{\ge 0} \to \R_{\ge 0}$ for all time: 
\begin{align}\label{eq:nonlinear_simple}
\dot x(t) &= x(t)(g_1(t) - g_2(t)x(t)^m), \quad \quad (m \in \Z_{> 0}). 
\end{align}
Then $x(t)$ converges to $\left(g_1^*/g_2^*\right)^{1/m}$ at a rate of at least
\begin{align}
\min\{\rho_{g_1}, \rho_{g_2}, mg_1^*\}. 
\end{align}
\end{lemma}

The next lemma helps establish rates of convergence for related functions.

\begin{lemma} \label{lem:lipschitz}
     Suppose $z(t) \to z^*$ at speed at least $\rho_z$ in the sense of Definition~\ref{def:speed}.
    Let $f$ be locally Lipschitz continuous on some open interval $U$ containing $z^*$. 
    Then $f(z(t)) \to f(z^*)$ at speed at least $\rho_z$. 
\end{lemma}
\begin{proof}
Since $z(t)\to z^*$ and $U$ is open, there exist $T>0$ and a compact interval
$I\subset U$ containing $z^*$ such that $z(t)\in I$ for all $t\ge T$.
Since $f$ is locally Lipschitz on $U$, it is Lipschitz on the compact interval $I$,
so there exists a constant $K>0$ such that for all $t\ge T$,
we have
\[
|f(z(t))-f(z^*)|\le K|z(t)-z^*|.
\]

If $f(z(t))=f(z^*)$ eventually, the conclusion is immediate. Otherwise, taking logarithms in the preceding inequality and dividing by $t$ gives
\[
\limsup_{t\to\infty}\frac{\ln|f(z(t))-f(z^*)|}{t}
\le
\limsup_{t\to\infty}\frac{\ln|z(t)-z^*|}{t},
\]
since $(\ln K)/t\to0$. Multiplying by $-1$ and applying Definition~\ref{def:speed} proves that $f(z(t))$ converges to $f(z^*)$ at rate at least $\rho_z$.
\end{proof}

\section{An explicit reaction network for System~5}\label{app:system5}

For concreteness we write out System~5 (Section~\ref{sec:log}) as an explicit mass-action
reaction network end to end. Recall that System~5 computes $\ln a^*$ for $a^*\ge 1$ in four
stages: generate the constant $e$, multiply the input by $e$, apply the logarithm module \eqref{eq:massactionforlog}, and subtract~$1$. We list the species and reactions stage by stage; all rate
constants are~$1$, and initial conditions are as indicated. Throughout, $A$ is the input
species with $a(t)\to a^*\ge 1$, and $I$ is a species held at the constant value $1$ (e.g.\
an input species with $i(0)=1$ and no consuming reactions).

\medskip\noindent\textbf{Stage 1: generate the constant $e$.}
Running the exponential module \eqref{eq:exponential} with the constant input $I$
(so $a\equiv 1$) produces a species $E$ with $e(t)\to e^{1}=e$:
\[
I\to I+Z_E,\qquad Z_E\to 0,\qquad I+E\to I+2E,\qquad Z_E+E\to Z_E,
\]
with $\dot z_E = i - z_E = 1 - z_E$ and $\dot e = e\,(i - z_E)$, initialized at
$e(0)=1$, $z_E(0)=0$, so that $e(t)=\exp(1-e^{-t})\to e$.

\medskip\noindent\textbf{Stage 2: multiply the input by $e$.}
The product $P=e\cdot a$ is formed by the multiplication module
\[
A+E\to A+E+P,\qquad P\to 0,\qquad \dot p = e\,a - p,
\]
so that $p(t)\to e\,a^*\ge e$. (This is the multiplication primitive of
\cite{anderson2025arithmetic}, included here explicitly.)

\medskip\noindent\textbf{Stage 3: logarithm of $P$.} Applying \eqref{eq:massactionforlog} with input $P$ yields a species $X$ with
$x(t)\to\ln(e\,a^*)=1+\ln a^*$:
\[
P+Z+X\to P+2Z+X,\qquad 2Z+X\to Z+X,\qquad P+X\to P+2X,\qquad Z+X\to Z,
\]
with $\dot z=(p-z)\,x z$ and $\dot x=(p-z)\,x$, initialized so that $x(0)=\ln z(0)$ and
$z(0)\ge e$. Because $p(t)\to e\,a^*\ge e$ at input-independent speed, Corollary~\ref{lem:massaction_log_eventual} applies and $x(t)\to 1+\ln a^*$ at input-independent speed.

\medskip\noindent\textbf{Stage 4: subtract $1$.}
We realize the final subtraction explicitly in dual rail using
\[
X\to X+W_p,\qquad I\to I+W_n,\qquad W_p\to0,\qquad W_n\to0,
\qquad W_p+W_n\to0.
\]
The designated output is the difference $w_p-w_n$. Since $i(t)\equiv1$, the corresponding
mass-action equations satisfy
\[
\frac{d}{dt}(w_p-w_n)=x(t)-1-(w_p-w_n).
\]
Thus, as $x(t)\to x^*=1+\ln a^*$, the dual-rail output converges to
\[
w_p(t)-w_n(t)\longrightarrow x^*-1=\ln a^*.
\]
By the linear tracking lemma from \cite{anderson2025arithmetic}, this stage preserves the
input-independent rate lower bound up to the intrinsic rate~$1$ of the tracking equation.

\medskip\noindent
Composing Stages~1--4 gives a single mass-action network whose designated dual-rail output
converges to $\ln a^*$ at input-independent speed. The full-domain System~6 is obtained
analogously, with a reciprocal stage and a maximum-with-$e$ stage (both from
\cite{anderson2025arithmetic}) inserted before Stage~3 and the output recorded in dual rail.

\subsection*{Acknowledgments}

DFA gratefully acknowledges financial support from the Office of the Vice Chancellor for Research at UW–Madison, with funding from the Wisconsin Alumni Research Foundation; from the Trustees of the William F. Vilas Estate; and from NSF grant DMS-2051498.
BJ is grateful for support from AMS-Simons Research Enhancement Grants for Primarily Undergraduate Institution (PUI) and NSF DMS-2051498. 
BJ and TDN gratefully acknowledge POSTECH and CM2LA for facilitating research discussions. 
We gratefully acknowledge Banff International Research Station for Mathematical Innovation and Discovery for hosting the authors at Research in Teams (Mathematical foundations of chemical computing (26rit047)).

\bibliographystyle{unsrt}
\bibliography{transcend}

\end{document}